\documentclass[preprint,12pt]{elsarticle}

\usepackage{amssymb}
\usepackage{amsmath}
\usepackage{amsthm}

\journal{Topology and its Aplications}

\usepackage[unicode, pdftex]{hyperref}

\begin{document}
\sloppy

\theoremstyle{theorem}
\newtheorem{theorem}{Theorem}[section]
\newtheorem{lemma}[theorem]{Lemma}
\newtheorem{proposition}[theorem]{Proposition}
\newtheorem{corollary}[theorem]{Corollary}
\newtheorem{remark}[theorem]{Remark}

\theoremstyle{definition}
\newtheorem{definition}{Definition}
\newtheorem{example}{Example}
\newtheorem{problem}{Problem}

\begin{frontmatter}


 
\title{Compact subspaces of the space of separately continuous functions with the cross-uniform topology\tnoteref{Grant}}
\tnotetext[Grant]{
	The research activities co-financed by the funds granted under the Research
	Excellence Initiative of the University of Silesia in Katowice
	}

\author[us,chnu]{Oleksandr Maslyuchenko\corref{CorrAuthor}} 
\ead{ovmasl@gmail.com}
\cortext[CorrAuthor]{Corresponding author}
\author[chnu]{Vadym Myronyk}
\ead{v.myronyk@chnu.edu.ua}
\author[chnu]{Roman Ivasiuk}
\ead{ivasiuk.roman@chnu.edu.ua}
\affiliation[us]
{organization={Institute of Mathematics, University of Silesia in Katowice},
addressline={Bankowa 12}, 
            city={Katowice},
            postcode={40-007}, 
            country={Poland}}
\affiliation[chnu]
{organization={Department of Mathematics and Informatics, Yuriy Fedkovych Chernivtsi National University},
	addressline={Kotsiubynskoho  2}, 
	city={Chernivtsi},
	postcode={58012}, 
	country={Ukraine}}
\begin{abstract}
We consider two natural topologies on the space $S(X\times Y,Z)$ of all separately continuous  functions defined on the product of two topological spaces $X$ and $Y$ and ranged into a topological or metric space $X$. These topologies are the cross-open topology and the cross-uniform topology. We show that these topologies coincides if $X$ and $Y$ are pseudocompacts and $Z$ is a metric space. We prove that a compact space $K$  embeds into $S(X\times Y,Z)$ for infinite compacts $X$, $Y$ and a  metrizable space $Z\supseteq\mathbb{R}$ if and only if  the weight of $K$ is less than the sharp cellularity of both spaces $X$ and $Y$.
\end{abstract}




\end{frontmatter}



\section{Introduction}

In \cite{VMa} the authors proposed a natural topologization of the space  of all separately continuous functions $s:[0;1]^2\to \mathbb{R}$ which was called the \textit{topology of the layered uniform convergence}. This topology can be considered on the space $S=S(X\times Y,Z)$ of all separately continuous functions ${s:X\times Y\to Z}$ for any topological spaces $X$ and $Y$ and a metric space $(Z,d)$. A base of this topology is given by the sets
$\big\{t\in S:\forall p\in\mathrm{cr}E\ \big|\  d\big(s(p),t(p)\big)<\varepsilon\big\}$,
where $E$ is a finite subset of $X\times Y$,  $\mathrm{cr}E=(X\times\mathrm{pr}_Y(E))\cup(\mathrm{pr}_X(E)\times Y)$ is the \textit{cross} of the set $E$, $\varepsilon>0$ and $s\in S(X\times Y)$. We call this topology the \textit{cross-uniform topology}. Another natural topology on $S$ is generated by the subbase consisting of the sets  $\big\{s\in S:s(A)\subseteq W\big\}$, where $A=\overline{G}\cap C$, $G$ is open in $C=\mathrm{cr}\{p\}$, $p\in X\times Y$ and $W$ is open in a topological space $Z$. It is called the \textit{cross-open topology} and we always endow the space $S$ with this topology. We prove that these two topologies coincide in the case where $X$ and $Y$ are pseudocompacts and $d$ is a metric generating the topology of a metrizable space $Z$.

If $X$ and $Y$ are compacts then $S(X\times Y)=S(X\times Y,\mathbb{R})$ is a topological vector space.
In \cite{VMa} it was  proved that $S([0,1]^2)$ is a separable non-metrizable complete topological vector space, and the authors asked about the other properties of this space.
In \cite{VMaM1,VMaM2,VMaM3}  it was proved that    $S(X\times Y)$ is a meager, complete, barreled and bornological topological vector space for any compacts $X$ and $Y$  without isolated points.

Another intrigued question on the space $S(X\times Y)$ is the problem of description of compact subspaces of $S(X\times Y)$ for any compacts $X$ and $Y$.
Compact subspaces of $B_1(X)$ (= the space of all Baire one functions with the pointwise topology) are the so-called \textit{Rosenthal compacts} if $X$ is a Polish space.
Since $S([0,1]^2)\subseteq B_1([0,1]^2)$ and every Baire one function on the diagonal $\Delta=\{(x,x):x\in[0;1]\}$ can be extended to a separately continuous function on $[0;1]^2$, we expected the appearance of some Rosenthal type compacts. But it turns out that the topological structure of compact subspaces
of $S(X\times Y)$ is simpler. For example, we prove that only a metrizable compact embeds into $S([0;1]^2)$. Moreover, we characterize compact subspaces of the space $S(X\times Y, Z)$.
Let $w(X)$ denote the weight of a topological space $X$
and  let $c(X)$ denote  the cellularity of $X$. The \textit{sharp cellularity} $c^\sharp(X)$ is defined by
$c^\sharp(X)=\sup\big\{|\mathcal U|^+:\mathcal U\mbox{ is
	a disjoint family of open sets in }X\big\}.$
In this paper we prove the following result: a compact space $K$  embeds into $S(X\times Y,Z)$ for infinite compacts $X$, $Y$ and a metrizable space $Z\supseteq\mathbb{R}$ if and only if  ${w(K)<\min\{c^\sharp(X),c^\sharp(Y)\}}$. The previous versions of this result were announced in \cite{Mconf, MIv1,MIv2}.

\section{Topologizations of the space of separately continuous functions}

Let $P$ be a set, $\alpha\subseteq 2^P$ and $Z$ be a topological space. Consider a set $A\in\alpha$ and an open set $W$ in $Z$. Denote
${\widetilde{O}_{A,W}=\big\{u\in Z^P:u(A)\subseteq W\big\}}.
$
The \mbox{\textit{$\alpha$-open topology}} is, by definition, the topology
 on $Z^P$ which is generated by the subbase 
${\big\{\widetilde O_{A,W}:A\in\alpha,\  W \text{ is an open subset of }Z \big\}}$.

Another natural method of topologization of the space of  functions concerns  the case where $Z$ is a metric space. So, let $(Z,d)$ be a metric space. The open and closed balls in this space we denote $B(z,\varepsilon)$ and $B[z,\varepsilon]$ respectively.
Consider $A\in\alpha$, $u\in Z^P$, ${\varepsilon>0}$. Denote 
$
{\widetilde B_A(u,\varepsilon)=\big\{v\in Z^P:d\big(u(x),v(x)\big)<\varepsilon\text{ for any }x\in A\big\}.}
$
The \textit{$\alpha$-uniform topology} is, by definition, the topology on $P^Z$ which is generated by the neighborhood subbase 
$
{\big\{\widetilde B_A(u,\varepsilon):A\in\alpha, \varepsilon>0 \big\}.}
$
This topology is generated by the family of extended pseudometrics ${\varrho_A:Z^P\times Z^P\to[0;+\infty]}$, 
${\varrho_A(u,v)=\sup\limits_{p\in A}d\big(u(p),v(p)\big)}$, $u,v\in Z^P$,
 $A\in\alpha$.

Suppose that $X$, $Y$ and $Z$ are topological spaces and $P=X\times Y$ is the product of topological spaces $X$ and $Y$. Let $f:P\to Z$ be a function and the functions $f^x:Y\to Z$ and $f_y:X\to Z$ be given by $f^x(y)=f_y(x)=f(x,y)$ for any $x\in X$ and $y\in Y$. 
A function $f$ is called \textit{separately continuous} if $f^x$ and $f_y$ are continuous for any $x\in X$ and $y\in Y$. The \textit{separate continuity} of a function $f:\prod_{i=1}^n X_i\to Z$ of several variables  means the continuity of the functions $f_{a,k}:\prod_{i\ne k}X_i\to Z$, $f_{a,k}(t)=f(a_1,\dots,a_{k-1},t,a_{k+1},\dots, a_n)$, $t\in X_k$ for any $a=(a_i)_{i=1}^n\in \prod_{i=1}^nX_i$ and $k=1,2,\dots, n$. 

Denote $S(X\times Y, Z)$ the set of all separately continuous functions ${s:X\times Y\to Z}$. Let $\mathrm{pr}_X:P\to X$ and $\mathrm{pr}_Y:P\to Y$ be the coordinate projections. That is $\mathrm{pr}_X(p)=x$ and $\mathrm{pr}_Y(p)=y$ for any $p=(x,y)\in P$. By the \textit{cross of a set} $E\subseteq P$ we mean the set
$$
\mathrm{cr}\,E=\big(\mathrm{pr}_X(E)\times Y\big)\cup\big(X\times \mathrm{pr}_Y(E)\big).
$$
Denote
\begin{align*}
	\mathcal{C}r(P)&=\Big\{\mathrm{cr}\{p\}:p\in P\Big\}\ \text{ and }\\
	\mathcal{C}ro(P)&=\Big\{\overline{G}\cap C:\text{$C\in\mathcal{C}r(P)$, $G$ is an open  set in $C$}\Big\}
\end{align*}
By the \textit{cross-open topology} we mean the restriction of the $\mathcal{C}ro(P)$-open topology to $S(P,Z)$. The space $S(P,Z)$ with the cross-open topology we denote by $S_{cro}(P,Z)$. Let
\begin{align*}
	O_{A,W}&=\widetilde{O}_{A,W}\cap S(P,Z)=\Big\{s\in S(P,Z):s(A)\subseteq W\Big\},\\
	\mathcal{P}_{cro}&=\Big\{O_{A,W}:A\in\mathcal{C}ro(P), W \text{ is an open subset of }Z\Big\}\ \ \text{and}\\
	\mathcal{B}_{cro}&=\Big\{\bigcap_{k=1}^n O_k:n\in\mathbb{N}\text{ and }O_1,O_2,\dots,O_n\in\mathcal{P}_{cro}\Big\}.
\end{align*}
Thus, $\mathcal{B}_{cro}$ is a base of the cross-open topology.

Let $(Z,d)$ be a metric space.   By the \textit{cross-uniform topology} we mean the restriction of the $\mathcal{C}r(P)$-uniform topology to $S(P,Z)$. The space $S(P,Z)$ with the cross-uniform topology we denote by $S_{cru}(P,Z)$. Set 
\begin{align*}
	B_A(s,\varepsilon)&=\widetilde B_A(s,\varepsilon)\cap S(P,Z)\\
	&=\Big\{t\in S(P,Z):d\big(s(x),t(x)\big)<\varepsilon\text{ for any }x\in A\Big\},\\
	\mathcal{P}_{cru}(s)&=\Big\{B_A(s,\varepsilon):A\in\mathcal{C}r(P), \varepsilon>0 \Big\}
	\ \ \text{and}\\
	\mathcal{B}_{cru}(s)&=\Big\{\bigcap_{k=1}^n B_k:n\in\mathbb{N}\text{ and }B_1,B_2,\dots,B_n\in\mathcal{P}_{cru}(s)\Big\}
\end{align*}
for any $s\in S(P,Z)$. Thus, $\mathcal{B}_{cru}(f)$ is a neighborhood base of the cross-uniform topology. Observe that the cross-uniform topology is generated by the family of extended pseudometrics $d_c:S(P,Z)\times S(P,Z)\to [0;+\infty]$,
\begin{eqnarray}\label{equ:metric_dc}
d_c(s,t)&=&\sup\limits_{p\in\mathrm{cr}\{c\}}d\big(s(p),t(p)\big)\\
&=&\nonumber\sup\limits_{x=a\text{ or }y=b}d\big(s(x,y),t(x,y)\big)\\
&=&\nonumber\max\Big\{\sup\limits_{x\in X}d\big(s(x,b),t(x,b)\big),\sup\limits_{y\in Y}d\big(s(a,y),t(a,y)\big)\Big\},
\end{eqnarray}
where $s,t\in S(P,Z)$ and $c=(a,b)\in P$.

We prove below that in the case where $X$ and $Y$ are pseudocompact and $Z$ is a metric space these two topologies coincide. And so, in this case we will omit the indexes in the notation of these spaces and denote the space of all separately continuous functions equipped with this topology by $S(X\times Y, Z)$.     
\begin{proposition}\label{prop:cru=cro}
	Let $X$ and $Y$ be pseudocompact spaces, $Z$ be a metrizable space and $d$ be a metric which generates the topology of $Z$. Then the cross-uniform topology on $S(X\times Y,Z)$ coincides with the cross-open topology, that is $S_{cru}(X\times Y,Z)=S_{cro}(X\times Y,Z)$. 
\end{proposition}
\begin{proof}
	Firstly, we prove that the cross-open topology is weaker than the cross-uniform topology. Let $P=X\times Y$, $c=(a,b)\in P$, $C=\mathrm{cr}\{c\}$, $G$ be an  open set in $C$, $K=\overline{G}\cap C$ and $W$ be an open subset of $Z$. Let us prove that ${O=O_{K,W}}$ is open in $S_{cru}(P,Z)$. Consider $s\in O$. Since $s$ is separately continuous, the restriction $s|_{C}$ is continuous. Observe that $C$ is pseudocompact and $K$ is the closure of an open set $G$ in $C$. By \cite[3.10.23]{En} one can easily prove that $K$ is a pseudocompact. Therefore, $s(K)$ is pseudocompact subset of a metrizable space $Z$. Then \cite[3.10.21, 4.1.17]{En} implies the compactness of $s(K)$. But $s(K)\subseteq W$. Thus, there is $\varepsilon>0$ such that 
	$B\big(s(p),\varepsilon\big)\subseteq W$ for any $p\in K$. Denote $U=B_C(s,\varepsilon)$ and prove that $U\subseteq G$. Let $t\in U$. Then $d\big(s(p),t(p)\big)<\varepsilon$ for any $p\in C$. Hence, $t(p)\in B\big(s(p),\varepsilon\big)\subseteq W$ for any $p\in K$. Therefore, $t(K)\subseteq G$. Consequently, $t\in O$. Thus, $G$ is open in $S_{cru}(P,Z)$. 
	
	Secondly, we prove that the cross-uniform topology is weaker than the cross-open topology. Let $c=(a,b)\in P$, $C=\mathrm{cr}\{c\}$, $s\in S(P,Z)$, $\varepsilon>0$ and $U=B_C(s,\varepsilon)$. Since $X$ and $Y$ are pseudocompact spaces, we conclude that $C$ is pseudocompact as well. But the separate continuity of $s$ implies that the restriction $s|_{C}$ is continuous. Therefore, $s(C)$ is a pseudocompact subset of $Z$. Then $s(C)$ is compact by \cite[3.10.21, 4.1.17]{En}. Set $B_p=B\big(s(p),\frac{\varepsilon}{3}\big)$ for any $p\in C$. Then $\big\{B_p:p\in C\big\}$ is an open covering of the compact set $s(C)$. Hence, there are $n\in\mathbb{N}$ and $p_1,p_2,\dots,p_n\in C$ such that $s(C)\subseteq\bigcup\limits_{i=1}^nB_{p_i}$. Put 
	$G_i=C\cap s^{-1}\big(B_{p_i}\big)$ and $K_i=\overline{G}_i\cap C$
	for any $i=1,2,\dots, n$. Since $s|_C$ is continuous, $G_i$ is open in $C$. Therefore, $K_i\in\mathcal{C}ro(P)$ and $s(K_i)\subseteq \overline{s(G_i)}\subseteq \overline{B}_{p_i}$ for any $i=1,2,\dots, n$. Let $W_i=B\big(s(p_i),\frac{\varepsilon}{2}\big)$ and $O_i=O_{K_i,W_i}$ for any $i=1,2,\dots, n$. Then the set $O=\bigcap\limits_{i=1}^n O_i$ is open in the cross-open topology. Since $s(K_i)\subseteq \overline{B}_{p_i}\subseteq W_i$ for any $i$, we conclude $s\in O$. So, $O$ is an open neighborhood of $s$ in the cross-open topology. To prove $O\subseteq U$ we pick $t\in O$. Then $t(K_i)\subseteq W_i$ for any $i$. On the other hand $s(K_i)\subseteq W_i$ as well. Consider $p\in C$. Since $s(p)\in s(C)\subseteq\bigcup\limits_{i=1}^nB_{p_i}$, there is $i$ with $s(p)\in B_{p_i}$. Then $p\in K_i$. Therefore,  $t(p)\in t(K_i)\subseteq W_i$ and $s(p)\in B_{p_i}\subseteq W_i$. Consequently, 
	$$d\big(s(p),t(p)\big)\le d\big(s(p),s(p_i)\big)+d\big(s(p_i),t(p)\big)<\tfrac{\varepsilon}{2}+\tfrac{\varepsilon}{2}=\varepsilon.$$
	Hence, $s\in U$. Thus, we have proved the inclusion $O\subseteq U$ and so $U$ is open in the cross-open topology.
\end{proof}

\section{The evaluation function and meagerness of  $S_{cro}(X\times Y,Z)$}

\begin{proposition}\label{prop:sep_cont_of_evaluation_function}
	Let $X$, $Y$ and $Z$ be topological spaces such that $Z$ is regular, $S=S_{cro}(X\times Y, Z)$, $\widetilde{X}=X\times S$, $\widetilde{Y}=Y\times S$. Define the evaluation functions $e:X\times Y\times S\to Z$, $e_1:\widetilde{X}\times Y\to Z$ and $e_2:X\times \widetilde{Y}\to Z$ by setting
	$e(x,y,s)=e_1(\tilde x,y)=e_2(x,\tilde y)=s(x,y),
	$
	where $x\in X$, $y\in Y$, $s\in S$, $\tilde x=(x,s)$ and $\tilde y=(y,s)$. Then $e$, $e_1$, $e_2$ are separately continuous functions. 
	
	Moreover, if $X$ and $Y$ are completely regular spaces without isolated points and $Z$ is a nontrivial Hausdorff topological vector space then  $e$, $e_1$, $e_2$ are everywhere discontinuous functions.
\end{proposition}
\begin{proof} Obviously, the separate continuity of $e_1$ (or $e_2$) implies that $e$ is separately continuous as well. So, we prove the separate continuity of $e_1$ only. Let $x_0\in X$, $y_0\in Y$, $s_0\in S$, $\tilde x_0=(x_0,s_0)$ and $z_0=e_1(\tilde x_0,y_0)=s_0(x_0,y_0)$. Consider a neighborhood $W_0$ of $z_0$. The regularity of $Z$ implies that there is a neighborhood $W$ of $z_0$ with $\overline{W}\subseteq W_0$. Set $p_0=(x_0,y_0)$ and $C=\mathrm{cr}\{p_0\}$. Since $s_0$ is separately continuous, there are open sets $U$ in $X$ and $V$ in $Y$  such that $x_0\in U$, $y_0\in V$ and  $s_0(G)\subseteq W$ where $G=(U\times V)\cap C$. Set $A=\overline{G}\cap C$. Then $A\in\mathcal{C}ro(P)$ and $s_0(A)\subseteq \overline{W}\subseteq W_0$. Put $O=O_{A,W_0}$. So, $O$ is an open neighborhood of $s_0$ in $S$. Set $\widetilde{U}=U\times O$. Consider $\tilde x=(x,s)\in \widetilde{U}$ and $y\in V$. Therefore, $s\in O$ and then  $s(A)\subseteq W_0$. Hence, $e_1(\tilde x, y_0)=s(x,y_0)\in s(A)\subseteq W_0$ and $e_1(\tilde x_0,y)=s_0(x_0,y)\in s_0(A)\subseteq W_0$. Thus, $e_1$ is separately continuous at $(\tilde x_0,y_0)$.
	
	Let us pass to the proof of the second part of the theorem. Consider a point $p_0=(x_0,y_0)\in P$ and $s_0\in S$. We are going to prove that $e$ is discontinuous at $q_0=(x_0,y_0,s_0)$. Let $z_0=e(q_0)=s_0(x_0,y_0)$, $z_1\in Z\setminus\{z_0\}$ and $W=Z\setminus\{z_1\}$. Then $W$ is a neighborhood of $z_0$. Consider a neighborhood $H$ of $q_0$. Therefore, there are open neighborhoods $U$ of $x_0$, $V$ of $y_0$ and a basic neighborhood $O\in\mathcal{B}_{cro}$ of $s_0$ such that $U\times V\times O\subseteq H$. Let $c_k=(a_k,b_k)\in P$, $C_k=\mathrm{cr}\{c_k\}$, $G_k$ be an open set in $C_k$, $F_k=\overline{G}_k\cap C_k$ and $W_k$ be an open set in $Z$ such that $O=\bigcap\limits_{k=1}^nO_{F_k,W_k}$. 
	
	Set $A=\{a_1,a_2,\dots,a_n\}$ and $B=\{b_1,b_2,\dots,b_n\}$. Since $X$ and $Y$ do not have isolated points, $A$ and $B$ are nowhere dense. Hence, there are points $x_1\in U\setminus A$ and $y_1\in V\setminus B$. Obviously, $F=\bigcup\limits_{k=1}^nF_k$ is a closed set and $p_1=(x_1,y_1)\notin F$. So, the complete regularity of $P$ yields that there is a continuous function $\varphi:P\to[0;1]$ such that $\varphi(p_1)=1$ and $\varphi(p)=0$ if $p\in F$. Denote $\bar z=z_1-s_0(p_1)$ and $s_1=s_0+\varphi\bar z$. Then $e(x_1,y_1,s_1)=s_1(p_1)=z_1\notin W$. But $s_1(p)=s_0(p)$ for $p\in F$. Therefore, $s_1(F_k)=s_0(F_k)\subseteq W_k$ for $k=1,2,\dots, n$. Hence, $s_1\in O$. Thus, $q_1\in(x_1,y_1,s_1)\in H$ and $e(q_1)\notin W$. So, $e(H)\not\subseteq W$ for any neighborhood $H$ of $q_0$. Consequently, $e$ is discontinuous at $q_0$. Then we conclude easily that $e_1$ and $e_2$ are discontinuous at $q_0$ as well. 
\end{proof}

Now we obtain some analogue of results from \cite{VMaM3}.

\begin{theorem}
	Let $X$ and $Y$ be completely regular spaces without isolated points  and $Z$ be a nontrivial metrizable topological vector space such that $X$ is $\alpha$-favorable and $Y$ is first countable at some point $b$. Then $S=S_{cro}(X\times Y, Z)$ is a meager space.
\end{theorem}
\begin{proof}
	Suppose that $S$ is not meager. Then by the Banach category theorem  there is an open residual Baire nonempty subspace $B$ of $S$.
	Let  $\widetilde{X}=X\times S$ and  $e_1:\widetilde{X}\times Y\to Z$ be defined by
	$e_1(\tilde x,y)=s(x,y),$
	where $\tilde x=(x,s)\in \widetilde{X}$ and $y\in Y$. By Proposition~\ref{prop:sep_cont_of_evaluation_function} $e_1$ is a separately continuous everywhere discontinuous function. Since $G=X\times B$ is an open subset of $\widetilde X$, the restriction $f=e_1|_G$ is a separately continuous everywhere discontinuous function as well. But the product of a Baire space and an $\alpha$-favorable space is Baire  \cite{Wh}. Hence, $G$ is a Baire space.  But $Y$ is first countable at some point $b$. So, the Calbrix-Troallic theorem  \cite{CaTr} yields that $f$ is continuous at $(\tilde a, b)$ for some point $\tilde a\in G$, which is impossible.
\end{proof}

\section{Reduction to the case of Eberlein compacts}

Denote by $C_p(X,Y)$ the space of all continuous functions $u:X\to Y$ equipped with the topology of pointwise convergent and $C_p(X)=C_p(X,\mathbb{R})$.  Recall that a compact space is called \textit{an Eberlein compact} if it is homeomorphic to a subspace of a Banach space with the weak topology (or, equivalently, to a subspace of $C_p(X)$ for a compact $X$).

\begin{proposition}\label{prop:EberleinCp(X,Y)}
	Let $X$ be a compact space, $Y$ be a 
	metrizable space
	and $K$ be a compact subspace of $C_p(X, Y)$. Then $K$ is an Eberlein compact.
\end{proposition}
\begin{proof}
	
	Let $d$ be a metric which generates the topology of $Y$. Following \cite[Proof of Step 1 of Theorem 1]{Chr} we define the closed subspace
	$$
	\Lambda=\Big\{\lambda\in [0;1]^Y:\forall y,y'\in Y\ \Big|\ \big|\lambda(y)-\lambda(y')\big|\le d(y,y')\Big\}
	$$
	of the Tychonoff cube $[0;1]^Y$. Therefore, $\Lambda$ is compact and so $Z=X\times \Lambda$ is compact as well. Define a mapping $\Phi:C_p(X,Y)\to [0;1]^Z$ by setting
	$$
	\Phi(u)(z)=\lambda\big(u(x)\big),\ \text{where }u\in C_p(X,Y)\text{ and }z=(x,\lambda)\in Z
	$$
	
	Firstly, we prove that $\Phi:C_p(X,Y)\to C_p(Z)$. Fix $u\in C_p(X,Y)$ and $c=(a,\alpha)\in Z$. Let us prove that $w=\Phi(u)$ is continuous at $c$. Consider, $\varepsilon>0$. Set $b=u(a)$ and $U=u^{-1}\big(B(b,\frac{\varepsilon}{2})\big)$. Since $u$ is continuous, $U$ is a neighborhood of $a$. Put $O=\Big\{\lambda\in\Lambda:\big|\lambda(b)-\alpha(b)\big|<\frac{\varepsilon}{2}\Big\}$. Then $O$ is a neighborhood of $\alpha$ in $\Lambda$. Therefore, $W=U\times O$ is a neighborhood of $c$ in $Z$. Consider a point $z=(x,\lambda)\in W$. Then $y=u(x)\in B(b,\frac{\varepsilon}{2})$ and $\big|\lambda(b)-\alpha(b)\big|<\frac{\varepsilon}{2}$. Consequently, 
	\begin{align*}
		\big|w(z)-w(c)\big|&=\big|\lambda(y)-\alpha(b)\big|\\
		&\le \big|\lambda(y)-\lambda(b)\big|+\big|\lambda(b)-\alpha(b)\big|\\
		&<d(y,b)+\tfrac{\varepsilon}{2}
		<\tfrac{\varepsilon}{2}+\tfrac{\varepsilon}{2}=\varepsilon.
	\end{align*}  
	Thus $w$ is continuous at $c$.
	
	Secondly, we prove that $\Phi:C_p(X,Y)\to C_p(Z)$ is continuous. Fix ${c=(a,\alpha)\in Z}$. Define a function $\varphi:C_p(X,Y)\to [0;1]$ by setting ${\varphi(u)=\Phi(u)(c)=\alpha\big(u(a)\big)}$, $u\in C_p(X,Y)$. It is enough to prove that $\varphi$ is continuous. Let $(u_m)_{m\in M}$ be a net in $C_p(X,Y)$ and $v\in C_p(X,Y)$ such that $u_m\to v$ in $C_p(X,Y)$. Then $u_m(a)\to v(a)$. Consequently, the continuity of $\alpha$ implies that   $\varphi(u_m)=\alpha\big(u_m(a)\big)\to\alpha\big(v(a)\big)=\varphi(v)$. Thus, $\varphi$ is continuous.
	
	Finally, let us prove that $\Phi$ is injective. Consider distinct points $u,u'\in C_p(X,Y)$. Then there is a point $a\in X$ such that $b=u(a)$ and $b'=u'(a)$ are distinct points in $Y$. Define $\alpha\in \Lambda$ by $\alpha(y)=\min\big\{1,d(b,y)\big\}$, $y\in Y$. Since $d(b,b')>0$, $\alpha(b')>0=\alpha(b)$. Let $c=(a,\alpha)$. Therefore, 
	$$
	\Phi(u)(c)=\alpha\big(u(a)\big)=\alpha(b)< \alpha(b')=\alpha\big(u'(a)\big)=\Phi(u')(c).
	$$
	Consequently, $\Phi(u)\ne\Phi(u')$.
	
	Thus, we have proved that $\Phi:C_p(X,Y)\to C_p(Z)$ is a continuous injection.  Therefore, $K$ is homeomorphic to the compact subspace $\widetilde{K}=\Phi(K)$ of $C_p(Z)$ and, hence, $K$ is an Eberlein compact.
\end{proof}

\begin{lemma}\label{lem:GammaAlpha}
	Let $X$, $Y$ and $Z$ be Hausdorff compact spaces, $\alpha:X\to Z$ be a continuous function and $\gamma:X\to Y$ be a continuous surjection such that
	\begin{equation}\label{equ:GammaAlpha}
		\forall x,x'\in X\ \big|\ \gamma(x)=\gamma(x')\Rightarrow \alpha(x)=\alpha(x').
	\end{equation}
	Then there is a unique continuous function $\beta:Y\to Z$ such that $\alpha=\beta\circ \gamma$.
\end{lemma}
\begin{proof} \textit{Existence.}
	Let $\xi:Y\to X$ be a right inversion of $\gamma$, that is $\gamma\big(\xi(y)\big)=y$ for any $y\in Y$. Define $\beta=\alpha\circ\xi$. By (\ref{equ:GammaAlpha}) we conclude that $\alpha=\beta\circ \gamma$. Indeed, if $x\in X$, $y=\gamma(x)$ and $x'=\xi(y)$ then $\gamma(x)=\gamma(x')=y$ and, hence, $\alpha(x)=\alpha(x')=\alpha\big(\xi(y)\big)=\beta(y)=\beta\big(\gamma(x)\big)$.
	
	\textit{Uniqueness.} Let $\beta,\beta_1:Y\to Z$ be functions with $\alpha=\beta\circ \gamma=\beta_1\circ \gamma$. Fix $y\in Y$. Since $\gamma$ is a surjection, there exists $x\in X$ with $y=\gamma(x)$. Therefore, $\beta(y)=\beta\big(\gamma(x)\big)=\alpha(x)=\beta_1\big(\gamma(x)\big)=\beta_1(y)$. Thus, $\beta=\beta_1$
	
	\textit{Continuity.} Consider a closed subset $C$ of $Z$. Then $A=\alpha^{-1}(C)$ is closed in $X$ and, hence, $A$ is compact. Therefore, $B=\gamma(A)$ is compact and, hence, $B$ is closed. Since $\gamma$ is surjective and 
	$A=\alpha^{-1}(C)=(\beta\circ \gamma)^{-1}(C)=\gamma^{-1}\big(\beta^{-1}(C)\big),$
	we conclude that $B=\gamma(A)=\gamma\Big(\gamma^{-1}\big(\beta^{-1}(C)\big)\Big)=\beta^{-1}(C)$. So, $\beta^{-1}(C)$ is closed. Thus, $\beta$ is continuous.
\end{proof}

\begin{lemma}\label{lem:ReductionToEberlein}
	Let $X$, $Y$, $K$ be compact spaces and $Z$ be a metrizable space such that $K\subseteq S(X\times Y, Z)$. Then there exists an Eberlein compact $\widetilde{X}$ such that $K$  embeds homeomorphically into $S(\widetilde{X}\times Y,Z)$ and $\widetilde{X}$ is a continuous image of $X$.
\end{lemma}
\begin{proof} Let $P=X\times Y$, $\widetilde{Y}=Y\times K$ and $f:X\times \widetilde{Y}\to Z$ be given by $f(x,\tilde y)=s(x,y)$, where $x\in X$ and $\tilde y=(y,s)\in \widetilde{Y}$. Therefore, $f$ is the restriction of the function $e_2$ from Proposition~\ref{prop:sep_cont_of_evaluation_function} to the product $X\times \widetilde{Y}$. So, $f$ is separately continuous. Define a mapping $\varphi:X\to C_p(\widetilde{Y},Z)$ by setting $\varphi(x)(\tilde y)=f(x,\tilde y)$, where $(x,\tilde{y})\in \widetilde{P}$. Since $f$ is separately continuous, it is easy to see that $\varphi$ is continuous. Set $\widetilde{X}=\varphi(X)$. Then $\widetilde{X}$ is a compact subspace of $C_p(\widetilde{Y},Z)$. Thus, it is an Eberlein compact by  Proposition~\ref{prop:EberleinCp(X,Y)}.
	
	Denote $\widetilde{P}=\widetilde{X}\times Y$ and $\widetilde{S}=S(\widetilde{P},Z)$. We are going to  prove that $K$ is homeomorphic to some subspace $\widetilde{K}$ of $\widetilde{S}$. Define $\psi:P\to \widetilde{P}$ by ${\psi(p)=\big(\varphi(x),y\big)}$, $p=(x,y)\in P$. Obviously $\psi$ is a continuous surjection. Let $\Psi:\widetilde{S}\to S$ be given by $\Psi(\tilde s)=\tilde s\circ \psi$, $\tilde s\in\widetilde{S}$. Since $\widetilde{P}=\psi(P)$, $\Psi$ is injective. Set $S_0=\Psi(\widetilde{S})$ and $\Phi=\Psi^{-1}:S_0\to \widetilde S$.
	
	Consider some metric $d$ which generates the topology of $Z$. Let $(d_{c})_{c\in P}$ and $(d_{\tilde c})_{\tilde c\in\widetilde{P}}$ be the families of pseudometrics on $S$ and $\widetilde{S}$ which are defined by (\ref{equ:metric_dc}). By Proposition~\ref{prop:cru=cro} the cross-uniform topology on $S$ (resp. $\widetilde{S}$) coincides with the cross-open topology and is generated by the family $(d_{c})_{c\in P}$ (resp. $(d_{\tilde c})_{\tilde c\in\widetilde{P}}$). Let $\tilde c=(\tilde a, b)\in \widetilde{P}$ and choose $a\in X$ with $\varphi(a)=\widetilde{a}$. Set $c=(a,b)$. So, $\psi(c)=(\varphi(a),b)=\tilde c$. Put $C=\mathrm{cr}\{c\}$ and $\widetilde{C}=\mathrm{cr}\{\tilde c\}$. Observe that
	\begin{align*}
		\widetilde{C}&=\Big\{(\tilde x,y)\in \widetilde{P}:\tilde x=\tilde a\text{ or }y=b\Big\}\\
		&=\Big\{\big(\varphi(x),y\big)\in \widetilde{P}:\varphi(x)=\varphi(a)\text{ or }y=b\Big\}\\
		&=\Big\{\big(\varphi(x),y\big)\in \widetilde{P}:x=a\text{ or }y=b\Big\}
		=\psi(C).
	\end{align*}
	Fix $s,t\in S_0$ and consider $\tilde s,\tilde t\in\widetilde{S}$ such that $s=\Psi(\tilde s)$ and $t=\Psi(\tilde t)$. Therefore, $s=\tilde s\circ \psi$ and $t=\tilde t\circ \psi$. Hence, 
	\begin{align*}
		d_c(s,t)&=\sup\limits_{p\in C}d\big(s(p),t(p)\big)
		=\sup\limits_{p\in C}d\Big(\tilde s\big(\psi(p)\big),\tilde t\big(\psi(p)\big)\Big)\\
		&=\sup\limits_{\tilde p\in \widetilde{C}}d\big(\tilde s(\tilde p),\tilde t(\tilde p)\big)
		=d_{\tilde c}(\tilde s,\tilde t).
	\end{align*}
	
	This proves that $\Phi$ and $\Psi=\Phi^{-1}$ are continuous with respect to the cross-uniform topology. Thus, $\Phi:S_0\to \widetilde{S}$ is a homeomorphism.
	
	It remains to prove that $K\subseteq S_0$. Consider $s\in K$. Fix $y\in Y$ and put $\tilde y=(y,s)$. Let $\alpha=s_y:X\to Z$, that is $\alpha(x)=s(x,y)$ for $x\in X$. Put $\gamma=\varphi$. Then $\gamma$ is a continuous surjection of $X$ onto $\widetilde{X}$. Let $x,x'\in X$ be points such that $\gamma(x)=\gamma(x')$. Therefore, 
	$$
	\alpha(x)=s(x,y)=\varphi(x)(\tilde y)=\gamma(x)(\tilde y)=\gamma(x')(\tilde y)=\varphi(x')(\tilde y)=s(x',y)=\alpha(x').
	$$
	So, (\ref{equ:GammaAlpha}) holds. Therefore, using Lemma~\ref{lem:GammaAlpha} with $Y=\widetilde{X}$ we conclude that there exists a continuous function $\beta:\widetilde{X}\to X$ with $\beta\circ\gamma=\alpha$. Set $\tilde s(\tilde x,y)=\beta(\tilde x)$ for any $\tilde x\in\widetilde{X}$. Consequently,  for $x\in X$ and $\tilde x=\gamma(x)=\varphi(x)$ we have
	$$\tilde s(\tilde x,y)=\beta(\tilde x)=\beta\big(\gamma(x)\big)=\alpha(x)=s(x,y).$$ 
	Thus, we have defined a function $\tilde s:\widetilde{P}\to Z$ such that $\tilde s(\tilde x,y)=s(x,y)$ for any $x\in X$, $y\in Y$ and $\tilde x=\varphi(x)$. So, 
	$\tilde s_y=s_y$ and $\tilde s^{\tilde x}=s^x$ are continuous. Thus, $\tilde s$ is a separately continuous function and so $\tilde s\in\widetilde{S}$. Observe that $\Psi(\tilde s)(x,y)=\tilde s(\varphi(x),y)=s(x,y)$ for any $(x,y)\in P$. Therefore, $s=\Psi(\tilde s)\in S_0$. Thus, we have proved the inclusion $K\subseteq S_0$. Hence, $K$ is homeomorphic to a subspace $\widetilde{K}=\Phi(X)$ of $\widetilde{S}$. 
\end{proof}

\section{The cellularity and the sharp cellularity of Eberlein compacts}

Let $w(X)$ denote the \textit{weight} of a topological space $X$, $d(X)$ denote the \textit{density} of $X$,
and  $c(X)$ denote the \textit{cellularity} of $X$, that is
\begin{align*}
	w(X)&=\min\{\frak |\mathcal{B}|: \mathcal{B} \text{ is a base of }X\},\\ 
	d(X)&=\min\{|A|:\overline{A}=X\}\ \ \ \ \ \text{and}\\
	c(X)&=\sup\Big\{|\mathcal U|:\mathcal U\mbox{ is
		a disjoint family of open sets in }X\Big\},
\end{align*}
where $|A|$ means the cardinality of a set $A$. The \textit{sharp cellularity}  is, by definition, the cardinal number 
$$c^\sharp(X)=\sup\Big\{|\mathcal U|^+:\mathcal U\mbox{ is
	a disjoint family of open sets in }X\Big\},$$
where $\frak m^+$ means the least cardinal number which is grater than  $\frak m$.
Remark that in \cite{Ju} the sharp cellularity is denoted by $\hat c(X)$. In \cite[4.1]{Ju} it was proved that $c^\sharp(X)=c(X)^+$ if $c(X)$ is either singular or successor cardinal number. The next lemma can be obtained from the previous fact. But the direct proof is also not so hard.

\begin{lemma}\label{lem:DisjointSystem}
	Let $X$ be a topological space, $\frak m>\aleph_0$, and $(\mathcal G_n)_{n=1}^\infty$ be a sequence of disjoint open systems $\mathcal G_n\not\ni\emptyset$ such that $\sup\limits_{n\in\mathbb N}|\mathcal G_n|=\frak m$. Then there exists a disjoint open system $\mathcal G\not\ni\emptyset$ such that $|\mathcal G|=\frak m$.
\end{lemma}
\begin{proof}
	If the cofinality $\mathrm{cf}(\frak m)>\aleph_0$ then there is $n\in\mathbb N$ with $|\mathcal{G}_n|=\frak m$ and we put $\mathcal{G}=\mathcal{G}_n$.
	Suppose that $\mathrm{cf}(\frak m)=\aleph_0$. Without lose of generality we may assume that $|\mathcal{G}_n|=\frak m_n$ is regular, $\aleph_0<\frak m_n<\frak m_{n+1}$  and $\bigcup \mathcal G_n$ is dense in $X$ for any $n$.
	
	
	We are going to construct sequences of disjoint open families $\mathcal U_n\not\ni\emptyset$ and  sets $W_n\in\mathcal U_n$ such that $\bigcup\mathcal U_n$ is dense in $X$, $\mathcal U_n\setminus \{W_n\}\subseteq \mathcal U_{n+1}$ and $|\mathcal U_n|=\frak m_n$. Let $\mathcal U_1=\mathcal G_1$. Suppose that we have already constructed $\mathcal U_n$ for some $n\in\mathbb N$. Set $\mathcal G_{n+1}(U)=\{G\in\mathcal G_{n+1}:G\cap U\ne\emptyset\}$. Then $\mathcal G_{n+1}=\bigcup_{U\in\mathcal U_n}\mathcal G_{n+1}(U)$. Since $|\mathcal U_n|=\frak m_n<\frak m_{n+1}$, $\frak m_{n+1}=|\mathcal G_{n+1}|\le\sum_{U\in\mathcal U_n}|\mathcal G_{n+1}(U)|\le \frak m_{n+1}$. So, by the regularity of  $\frak m_{n+1}$ there exists $W_n\in \mathcal U_n$ such that $|\mathcal G_{n+1}(W_n)|=\frak m_{n+1}$. Set $\mathcal U_{n+1}=(\mathcal U_n\setminus\{W_n\})\cup\{G\cap W_n:G\in \mathcal G_{n+1}(W_n)\}$. Thus, $\mathcal U_{n+1}$ and $W_n$ have the needed properties. 
	
	Let $\mathcal W=\{W_n:n\in\mathbb N\}$ and $\mathcal V_n=\mathcal U_n\setminus \mathcal W$ for any $n\in\mathbb N$. Obviously, $\mathcal V_n$ is a disjoint open family, $\mathcal V_n\subseteq \mathcal V_{n+1}$ and $|\mathcal V_n|=\frak m_n$   for any $n\in\mathbb N$. Thus, the family $\mathcal G=\bigcup_{n=1}^\infty \mathcal V_n$ is disjoint and $|\mathcal G|=\frak m$. 
\end{proof}

In \cite[Theorem 4.2]{BeRuWa} the authors proved that any Eberlein compact includes into $c_0(\Gamma)$ with the pointwise topology such that $|\Gamma|=c(X)$. Formally,  they proof is incomplete, because they proved that for any cardinal $\mu<\lambda=|\Gamma|$  there is an open disjoint family of the cardinality $\ge\mu$ only. But if $\lambda$ is regular then this is not enough to assert that $c(X)=\lambda$. Of course,  we can easily to complete the proof for a regular $\lambda$. The authors also remarked in \cite[Remark (a) after Theorem 4.2]{BeRuWa} that  one can prove that there is a disjoint open family of cardinality $\mu$. In the next proposition we do it directly using their method and our Lemma~\ref{lem:DisjointSystem}.

\begin{proposition}\label{prop:SharpCelullatityEberlein}
	Let $X$ be an Eberlein compact. Then $c^\sharp(X)=c(X)^+$ and   $c(X)=d(X)=w(X)$.
\end{proposition}
\begin{proof} 
 By the Amir-Lindersrauss theorem \cite[Theorem 1]{AL} we may assume that $X$ be a subspace of $c_0(T)$ with the pointwise topology. Denote 
${S(x)=\big\{t\in T:x(t)\ne 0\big\}}$, $S=\bigcup_{x\in X}S(x)$, 
$S_n(x)=\big\{t\in T:|x(t)|>\tfrac1n\big\}$ and $S_n=\bigcup_{x\in X}S_n(x)$
for any $x\in X$ and $n\in\mathbb N$. Evidently, $S_n(x)$'s are finite and $S=\bigcup\limits_{n=1}^\infty S_n$. Set $\frak m=|S|$. Observe that $X$ is homeomorphic with the subspace $X|_S=\big\{x|_S:x\in X\big\}$ of $\mathbb R^S$. Therefore $w(X)\le w(\mathbb R^S)=\frak m\cdot\aleph_0$. 

In the case where $\frak m\le\aleph_0$ we have that $X$ is metrizable and separable. Thus,  $c(X)=d(X)=w(X)\le \aleph_0$ and  $c^\sharp(X)= c(X)^+$. 

Suppose that $\frak m>\aleph_0$. Then $w(X)\le\frak m$. Set
$\mathcal S_n=\big\{S_n(x):x\in X\big\}$  and 
	$\mathcal M_n=\big\{M:\text{$M$ is a maximal element of $\mathcal S_n$}\big\}$.
Since $X$ compact subset of $c_0(T)$, $\mathcal S_n$ does not contain an infinite increasing chain. Therefore, each element of $\mathcal S_n$ includes in some element of $\mathcal M_n$.  Consequently, $S_n=\bigcup\mathcal M_n$. Denote $\frak m_n=|\mathcal M_n|$. 
Since any element of $\mathcal M_n$ is finite,  $\frak m_n=|S_n|$ if $S_n$ is infinite. Thus, $\sup\limits_{n\in\mathbb N}\frak m_n=\frak m$. Fix $n\in\mathbb N$ and define for any $M\in\mathcal M_n$
$$
G_M=\big\{x\in X:|x(t)|>\tfrac1n\text{ for any }t\in M\big\}.
$$
The maximality of elements of $\mathcal M_n$ yields $G_M\cap G_N=\emptyset$  for any distinct sets $M,N\in\mathcal M_n$. Therefore, the disjoint open family  $\mathcal G_n=\big\{G_M:M\in\mathcal M_n\}$ has the cardinality $\frak m_n$. By Lemma~\ref{lem:DisjointSystem} there is a disjoint open system $\mathcal G$ in $X$ such that $|\mathcal G|=\frak m$. Therefore, $\frak m\le c(X)\le d(X)\le w(X)\le \frak m$. Thus, these cardinals are equal and $c^\sharp(X)=\frak m^+$.
\end{proof}

\section{The weight of compact subspaces of $S(X\times Y,Z)$}
\begin{proposition}\label{prop:DensityOfCompactInCp(X,M)}
	Let $X$ be a topological space, $M$ be a metrizable space and $K$ be a compact subspace of $C_p(X,M)$. Then $w(K)\le d(X)\cdot \aleph_0$. In particular, if $d(X)\ge\aleph_0$ then $w(K)\le d(X)$.
\end{proposition}
\begin{proof}
	Let $A$ be a dense subset of $X$ such that $|A|=d(X)$. Let ${\Phi:C_p(X,M)\to M^A}$ be defined by $\Phi(u)=u|_A$, $u\in C_p(X,M)$. Obviously, $\Phi$ is a continuous injection. Consequently, $K$ is homeomorphic to a subspace $L=\Phi(K)$ of $M^A$. Let $\mathrm{pr}_a:M^A\to M$ be the coordinate projection and $L_a=\mathrm{pr}_a(L)$ for any $a\in A$. Therefore, $L_a$'s are metrizable compacts and, hence, $w(L_a)\le\aleph_0$. By \cite[2.3.13]{En} we obtain that $$\textstyle w(K)=w(L)\le w\Big(\prod_{a\in A}L_a\Big)\le |A|\cdot\aleph_0=d(X)\cdot \aleph_0.$$
	If $d(X)\ge\aleph_0$ then $w(K)\le d(X)\cdot \aleph_0=d(X)$.
\end{proof}
\begin{theorem}\label{thm:WeightOfCompactSubspaceS}
	Let $X$, $Y$ be infinite compacts, $Z$ be a 
	metrizable space and  $K$ be a compact subspace of $S(X\times Y,Z)$. Then $w(K)<\min\big\{c^\sharp(X),c^\sharp(Y)\big\}$.
\end{theorem}
\begin{proof} Set $\frak m=w(K)$. It is enough to prove that $\frak m< c^\sharp(X)$. By Lemma~\ref{lem:ReductionToEberlein} there is an Eberlein compact $\widetilde X$ such that $K$ embeds into $S(\widetilde{X}\times Y,Z)$ and $\widetilde{X}$ is a continuous image of $X$. Consequently, $c^\sharp(\widetilde{X})\le c^\sharp(X)$. So, it will be enough to prove that $\frak m<c^\sharp(\widetilde{X})$. Thus, without loss of generality we may assume that $X$ is an Eberlein compact.
	
	Let $P=X\times Y$, $d$ be a metric which generates the topology of $Z$ and $(d_c)_{c\in P}$ be defined by (\ref{equ:metric_dc}). Then the cross-uniform topology on the space $S=S(P,Z)$ is generated by this family of pseudometrics. Consider the metric space $M=C(Y,Z)$ with the supremum metric $$\varrho(u,v)=\sup\limits_{y\in Y}d\big(u(y),v(y)\big),\ \  u,v\in M.$$
	Let  $\Phi:S\to C_p(X,M)$ be given by $\Phi(s)(x)(y)=s(x,y)$, $x\in X$, $y\in Y$, $s\in S$. Let $c=(a,b)\in P$, $s,t\in S$ and set $u=\Phi(s)$, $v=\Phi(t)$. Then
	$$
	\varrho\big(u(a),v(a)\big)=\sup\limits_{y\in Y}d\big(u(a)(y),v(a)(y)\big)
	=\sup\limits_{y\in Y}d\big(s(a,y),t(a,y)\big)\le d_c(s,t).
	$$
	Therefore, $\Phi$ is a continuous injection. Consequently, $K$ is homeomorphic to the subspace $\widetilde K=\Phi(K)$ of $C_p(X,M)$. Thus,  Propositions~\ref{prop:DensityOfCompactInCp(X,M)} and \ref{prop:SharpCelullatityEberlein} now leads to $\frak m=w(\widetilde{K})\le d(X)=c(X)<c(X)^+=c^\sharp(X)$.
\end{proof}
\section{Characterization of compact subspaces of $S(X\times Y,Z)$}
\begin{theorem}\label{thm:CharacterizationOfCompactSubspaceS}
	Let $X$, $Y$ be infinite compacts and $Z$ be a 
	metrizable space  containing a homeomorphic copy of $\mathbb R$.  Then a compact space $K$ homeomorphically embeds  into $S(X\times Y, Z)$ if and only if  ${w(K)<\min\big\{c^\sharp(X),c^\sharp(Y)\big\}}$.
\end{theorem} 
\begin{proof}
	The necessity was proved in Theorem~\ref{thm:WeightOfCompactSubspaceS}. Let us prove the sufficiency. Let $\frak m=\min\big\{c^\sharp(X),c^\sharp(Y)\big\}$ and $I$ be a set such that $|I|=\frak m$. Since $X$ and $Y$ are infinite compacts, $\frak m\ge\aleph_0$. By the Tychonoff embedding theorem \cite[2.3.23]{En} $K$ embeds into the the Tychonoff cube $Q=[0;1]^I$.  Without loss of generality we may assume that $\mathbb R\subseteq Z$. Consequently, $S(P)\subseteq S(P,Z)$.  So, it is enough to prove that $Q$ embeds into $S=S(P)$.
	
	Since $|I|=\frak m<\min\big\{c^\sharp(X),c^\sharp(Y)\big\}$, there are disjoint  families $(U_i)_{i\in I}$  and $(V_i)_{i\in I}$ of non-empty open sets in $X$ and $Y$ respectively. Let $i\in I$. Fix some points $a_i\in U_i$, $b_i\in V_i$ and set $c_i=(a_i,b_i)$. By the complete regularity of $X\times Y$ there is a continuous function $f_i:X\times Y\to[0;1]$ such that $f_i(c_i)=1$  and $f_i(p)=0$ for any $p\in P\setminus W_i$. Define a mapping $\Phi:Q\to \mathbb{R}^P$ by setting
	$$
	\Phi(u)(p)=\sum\limits_{i\in I}u_if_i(p),\ \ u=(u_i)_{i\in I}\in Q,\ \  p\in P.
	$$
	Since $f_i$'s vanish outsides of $W_i$ and $(W_i)_{i\in I}$ is disjoint, $\Phi$ is well defined. 
	
	Let ${c=(a,b)\in P}$ and $C=\mathrm{cr}\{c\}$. Then there are $i_0,j_0\in I$ such that 
	$a\notin U_i$ for any $i\ne i_0$ and $b\notin V_j$ for any $j\ne j_0$. Consequently, $C\cap W_i=\emptyset$ for any $i\notin\{i_0,j_0\}$. Therefore, $\Phi(u)(p)=u_{i_0}f_{i_0}(p)+u_{j_0}f_{j_0}(p)$ for any $u\in Q$ and
	$p\in C$. Thus,  $\Phi:Q\to S$. To prove the continuity of $\Phi$ consider a net $v_m=(v_{m,i})_{i\in I}$ which is convergent to $u=(u_i)_{i\in I}$ in $Q$. Therefore, $\Phi(v_m)(p)=v_{m,i_0}f_{i_0}(p)+v_{m,j_0}f_{j_0}(p)$  and $\Phi(u)(p)=u_{i_0}f_{i_0}(p)+u_{j_0}f_{j_0}(p)$ for any $m\in M$ and $p\in C$. Since $f_{i_0}(p),f_{j_0}(p)\in[0;1]$, we conclude that 
	\begin{align*}
		d_c\big(\Phi(v_m),\Phi(u)\big)
		&=\sup\limits_{p\in C}\Big|\big(v_{m,i_0}-u_{i_0}\big)f_{i_0}(p)+\big(v_{m,j_0}-u_{j_0}\big)f_{j_0}(p)\Big|\\
		&\le\sup\limits_{p\in C}\Big(f_{i_0}(p)\big|v_{m,i_0}-u_{i_0}\big|+f_{j_0}(p)\big|v_{m,j_0}-u_{j_0}\big|\Big)\\
		&\le\big|v_{m,i_0}-u_{i_0}\big|+\big|v_{m,j_0}-u_{j_0}\big|\to 0
	\end{align*}
	Hence, $\Phi$ is continuous. But $\Phi$ is injective and $Q$ is compact. Thus, $Q$ is homeomorphic to the subspace $\Phi(Q)$ of $S$.
\end{proof}
\begin{corollary}\label{cor:CharacterizationOfCompactSubspaceS}
	Let $X$, $Y$ be infinite metrizable  compacts and $Z$ be a  
	metrizable space containing a homeomorphic copy of $\mathbb R$.  Then a compact space $K$ homeomorphically embeds  into $S(X\times Y, Z)$ if and only if  it is metrizable.
\end{corollary} 
\section{Open problems}
The compactness play the essential role in our constructions. So, the next problem seems quite difficult.  
\begin{problem}
	Let $X$, $Y$ be topological spaces and $Z$ be a topological (resp. metric) space. Describe a compact $K$ which  homeomorphically includes into $S_{cro}(X\times Y,Z)$  (resp.  $S_{cru}(X\times Y,Z)$).
\end{problem}
Theorem~\ref{thm:CharacterizationOfCompactSubspaceS} gives the answer to the previous problem in the case where $X$ and $Y$ are infinite compacts and $Z$ is a metrizable space containing a homeomorphic copy of $\mathbb R$. The following two questions are more concrete versions of the previous problem.  
\begin{problem}
	Let $K$ be a Rosenthal compact. Do there exist Polish spaces $X$ and $Y$ such that $K$ homeomorphically includes into $S_{cro}(X\times Y)$  (resp.  $S_{cru}(X\times Y)$)?
\end{problem}
\begin{problem}
	Do there exist Polish spaces  $X$ and $Y$ such that there is a non-metrizable compact subspace $K$  of $S_{cro}(X\times Y)$  (resp. $S_{cru}(X\times Y)$)?
\end{problem}
By Corollary~\ref{cor:CharacterizationOfCompactSubspaceS} Polish spaces in two previous problems should be non-compact.
\begin{problem}
	Does there exist a non-metrizable  space $Y$ such that an arbitrary compact $K$ is an Eberlein compact if and only if it homeomorphically includes into $C_p(X,Y)$ for some compact $X$? 
\end{problem}

By Proposition~\ref{prop:EberleinCp(X,Y)} a metrizable space $Y\supseteq \mathbb R$ has the property from the previous problem.
\section*{Acknowledgments}
The authors would like to appreciate Taras Banakh for the suggestion to use the sharp cellularity.




\end{document}